\documentclass[12pt,a4paper]{amsart}
\usepackage{graphicx,amssymb}
\input xy
\xyoption{all}
\usepackage[all]{xy}
\usepackage{hyperref}

\newtheorem{thm}{Theorem}[section]
\newtheorem{cor}[thm]{Corollary}
\newtheorem{lem}[thm]{Lemma}
\newtheorem{prop}[thm]{Proposition}
\theoremstyle{definition}
\newtheorem{defn}[thm]{Definition}
\newtheorem{rem}[thm]{Remark}

\numberwithin{equation}{section}

\newcommand{\QQ}{\mathbb Q}

\newcommand{\CC}{\mathbb C}
\newcommand{\PP}{\mathbb P}
\newcommand{\FF}{\mathbb F}

\newcommand{\lra}{\longrightarrow}

\newcommand{\ra}{\rightarrow}

\newcommand{\cA}{\mathcal{A}}
\newcommand{\cD}{\mathcal{D}}

\newcommand{\cJ}{\mathcal{J}}

\newcommand{\cR}{\mathcal{R}}

\newcommand{\cC}{\mathcal{C}}

\newcommand{\cM}{\mathcal{M}}

\newcommand{\cO}{\mathcal{O}}

\newcommand{\cS}{\mathcal{S}}

\newcommand{\cZ}{\mathcal{Z}}
\newcommand{\cX}{\mathcal{X}}
\newcommand{\cY}{\mathcal{Y}}

 \DeclareMathOperator{\Ker}{Ker}
  
\DeclareMathOperator{\Pic}{Pic}
 \DeclareMathOperator{\Nm}{{Nm}}
  \DeclareMathOperator{\Prym}{{Prym}}

 \DeclareMathOperator{\Stab}{Stab}

\DeclareMathOperator{\SL}{{SL}}

 \DeclareMathOperator{\Proj}{Proj}
\DeclareMathOperator{\Orb}{Orb}

\DeclareMathOperator{\supp}{{supp}}

\begin{document}

\title[ ]{ Pryms of non-cyclic triple coverings and log canonical models of the spin moduli space of genus 2}
\author{ Herbert Lange and  Angela Ortega}

\address{H. Lange \\ Department Mathematik der Universit\"at Erlangen \\ Germany}
\email{lange@mi.uni-erlangen.de}
              
\address{A. Ortega \\ Institut f\" ur Mathematik, Humboldt Universit\"at zu Berlin \\ Germany}
\email{ortega@math.hu-berlin.de}

\thanks{The second author was supported by Deutsche Forschungsgemeinschaft, SFB 647}

\begin{abstract} 
We show that the moduli spaces of non-cyclic triple covers of genus 2 curves and that of even spin curves of genus 2
are birationally isomorphic via the Prym map.  We describe the log canonical models of the moduli space of $S^+_2$  and 
use this to extend the above birational map. 
\end{abstract}

\maketitle

\section{Introduction}

In recent years,  the log minimal  model program has been carried out completely for the  moduli space 
$\overline{M}_g$ in low genera and in many cases these models have also a modular meaning. For instance, in
\cite{h} the Deligne-Mumford compactification and the GIT compactification of $M_2$ are realized as log minimal
models of $\overline{M}_2$. In this paper we study the log canonical models of the moduli space $\overline{S}^+_2$ of even spin
 curves of genus 2. Our initial motivation was  to find out if the birational morphism, induced by the Prym map,  between the 
moduli space $R_{2,3}^{nc}$  of non-cyclic \'etale triple coverings of genus 2 curves and $S^+_2$  (\cite{lo1}, \cite{lo2})
could be related to a new modular compactification of the spin moduli space.
In order to state our theorems let us recall some previous results. 
 
Let $f: Y \ra X$ be a non-cyclic \'etale 3-fold covering of a smooth projective curve $X$ of genus 2. 
Its Prym variety $P(f) := (\Ker \Nm f)^0$ is a Jacobian surface with principal polarization $\Xi$, giving rise to a 
map $Pr: R_{2,3}^{nc} \ra \cJ_2$  into the moduli space of 
canonically polarized Jacobian surfaces. In \cite{lo1} we showed that this map is finite of degree 
10 onto its image which is an open set in $\cJ_2$. We also proved that the Prym map is not surjective and 
determined its image explicitly. 

 In \cite{lo2} we extended this map to a proper map which is surjective onto the moduli space 
$\cA_2$  of principally polarized abelian surfaces.  To be more precise, we denote by 
$\mathfrak{S}_3 := \langle \sigma, \tau \;| \; \sigma^3 = \tau^2 = (\sigma \tau)^2 = 1 \rangle$ 
the symmetric group of order 6 and let 
${}_{\mathfrak{S}_3}{\overline M}_2$ denote the moduli space of admissible $\cS_3$-covers
of stable curves of genus 2, as defined in \cite{acv}. We consider the following open subspace 
$$
{}_{\mathfrak{S}_3}{\widetilde M}_2 := \left\{ [h: Z \ra X] \in {}_{\mathfrak{S}_3}{\overline 
M}_2 \;\left| \begin{array}{l}   p_a(Z) =7   \;  \mbox{and for any node}\;  z \in Z, \\
                                                                     \mbox{the                                                    
                                                                     stabilizer}                                  
                                                                     \Stab (z) \; \mbox{is of order 3}
                                                                     \end{array} \right. \right\}.
$$
For all $[h: Z \ra X] \in {}_{\mathfrak{S}_3}{\widetilde M}_2$ the Prym variety of 
$f: Y := Z/\langle \tau \rangle \ra X$ is a principally polarized abelian surface, independent of the 
involution $\tau \in \mathfrak S_3$. This defines a Prym map 
$Pr:  {}_{\mathfrak{S}_3}{\widetilde M}_2 \ra \cA_2$ which is proper, surjective and finite of 
degree 10.

The description of the fibres of the Prym map suggests that there is a close relation between the space
${}_{\mathfrak{S}_3}{\widetilde M}_2$ and the moduli space $S_2^+$ of even spin curves of genus 
2. We define the following open subspace
$$
{}_{\mathfrak{S}_3}{\widehat M}_2 := \left\{ [h: Z \ra X] \in {}_{\mathfrak{S}_3}{\widetilde 
M}_2 \;\left| \begin{array}{l} $X$ \; \mbox{irreducible}
                                                                     \end{array} \right. \right\}.
$$
Then our first result is\\

\noindent
{\bf Theorem 1.}
{\it There is a canonical isomorphism $\alpha: {}_{\mathfrak{S}_3}{\widehat M}_2 \ra S_2^+$.}\\

This is true even on the level of moduli stacks (see Theorem \ref{thm3.1}). We then consider the 
question whether one can extend this isomorphism to a regular map defined on the whole of 
${}_{\mathfrak{S}_3}{\widetilde M}_2$ with values in some compactification of $S_2^+$.
For this we work out the log canonical models of the compactification $\overline{S}_2^+$ of $S_2^+$, as defined by Cornalba in \cite{c}.  We denote by 
$\overline{\cS}_2^+$ (respectively $\cS_2^+$) the moduli stack corresponding to
$\overline{S}_2^+$ (respectively $S_2^+$) and by $\delta := \overline{\cS}_2^+ \setminus \cS_2^+$ the boundary divisor. 
The {\it log canonical models }  of  $\overline{\cS}_2^+$ with respect to $K_{\overline{\cS}_2^+} + 
\epsilon \delta$  are defined by
$$ 
\overline{\cS}_2^+(\epsilon):= \Proj \left( \bigoplus_{n\geq 0 } \Gamma (\overline{\cS}_2^+,  n(K_{\overline{\cS}_2^+} + 
\epsilon \delta )) \right), 
$$
for $\epsilon \in \QQ \cap [0,1]$.  The fact that sections of invertible sheaves on $\overline{\cS}_2^+$ are
pullbacks of sections of the corresponding sheaves on $\overline{S}_2^+$ implies that the log canonical model of
$\overline{\cS}_2^+$  with respect to $K_{\overline{\cS}_2^+} + \epsilon \delta$ can be identified with the 
log canonical model $\overline{S}_2^+(\epsilon)$ of $\overline{S}_2^+$ with respect to the corresponding divisor in 
$\overline{S}_2^+$ (see Corollary \ref{ramform}).

The following theorem (see the end of Sections 6 and 9) is analogous to the corresponding result of Hassett's for
$\overline{M}_2$ (see \cite[Theorem 4.10]{h}).
\\
\\
\noindent
{\bf Theorem 2.}
{\it Consider the log canonical model of $\overline{\cS}_2^+$ with respect to $K_{\overline 
{\cS}_2^+} + \epsilon \delta$, that is, the log canonical model of $\overline{S}_2^+$ with 
respect to the corresponding divisor in $\overline{S}_2^+$.
\begin{enumerate}
\item For $\epsilon > \frac{57}{25}$ we have $ \overline{S}_2^+(\epsilon) \simeq \overline{S}_2^+$.

\item For $\frac{49}{25} < \epsilon \leq \frac{57}{25}$ we have $ \overline{S}_2^+(\epsilon) \simeq \overline{S_2^+}^{inv}$,
where $\overline{S_2^+}^{inv}$ denotes the invariant theoretical compactification  (for the definition see Section 7).

\item For $\epsilon = \frac{49}{25}$ we get a point; the log canonical divisor fails to be effective for 
$\epsilon < \frac{49}{25}$.
\end{enumerate}}

The invariant theoretical compactification of $S_2^+$ is the $\Proj$ of certain ring of invariants arising from 
the parametrization of binary sextic forms together with a partition of the roots into two sets of three elements.  
Concerning the extension of the map $\alpha$ to the whole of ${}_{\mathfrak{S}_3}{\widetilde M}_2$ we finally obtain 
(see Propositions \ref{prop4.1} and \ref{prop10.1}):\\

\noindent
{\bf Theorem 3.} {\it
\begin{enumerate}
\item The construction defining the isomorphism $\alpha: {}_{\mathfrak{S}_3}{\widehat M}_2 \ra S_2^+$ does not extend to ${}_{\mathfrak{S}_3}{\widetilde M}_2$.
\item The isomorphism $\alpha: {}_{\mathfrak{S}_3}{\widehat M}_2 \ra S_2^+$ extends to a regular map ${}_{\mathfrak{S}_3}{\widetilde M}_2 \ra \overline{S_2^+}^{inv}$.
\end{enumerate}}

\hspace{0.3cm}
The paper is organized as follows. In Section 2 we recall the definitions of the main moduli stacks 
and spaces which are used in the sequel. In Section 2 the proof of Theorem 1 is given. In Section 4
we show Theorem 3 (1). In Section 5 we compute the nef cone of $\overline{S}_2^+$
and apply this in Section 6 to prove Theorem 2 (1). In Section 7 and 8 we study the 
invariant-theoretical compactification $\overline{S_2^+}^{inv}$ and the canonical map
$\overline{S}_2^+ \ra \overline{S_2^+}^{inv}$ and use this in Section 9 to give a proof of 
Theorem 2 (2) and 2 (3). Finally, in Section 10 we prove Theorem 3 (2).\\

\vspace{0.6cm}
We work over an algebrically closed field $k$ of characteristic zero. Moduli spaces are denoted by 
capital bold letters, the corresponding moduli stacks by the corresponding cursive letters. Divisors on 
a coarse moduli space are denoted by capital Latin letters,
divisors on a moduli stack by small Greek letters. We denote a divisor and its class in the rational 
Picard group by the same letter. \\

We thank I. Dolgachev for the hint leading to Remark \ref{rem7.2}. The second author is thankful to 
G. Farkas for helpful and stimulating discussions. 
                                         
\section{The moduli stacks}

\subsection{The stacks of $\mathfrak{S}_3$-coverings of genus-2 curves}

Let ${}_{\mathfrak{S}_3}\cM_2$ denote the moduli stack of \'etale Galois covers of smooth curves of genus 2 with Galois group the symmetric group $\mathfrak{S}_3$ of order 6 (\cite[Theorem 17.2.11]{acg}) and its compactification
${}_{\mathfrak{S}_3}{\overline \cM}_2$ of ${}_{\mathfrak{S}_3}\cM_2$ by admissible $\mathfrak{S}_3$-covers as constructed in 
\cite{acv} (see also \cite[Chapter 17]{acg}).

We consider the following open substack  ${}_{\mathfrak{S}_3}{\widehat \cM}_2$ of ${}_{\mathfrak{S}_3}{\overline \cM}_2$
associated to the functor \newline $F: Sch/\CC \ra  Ens$ defined by
$$
S \mapsto \left\{ [h:\cZ \ra \cX \; \mbox{over}\; S] \in {}_{\mathfrak{S}_3}{\overline \cM}_2 \;\left| \begin{array}{l}
                                                                    \mbox{for all}\;  s \in S, \; p_a(\cZ_s) = 7, \cX_s \; \mbox{is                                  
                                                                         irreducible,}\\ 
                                                                   \mbox{for any node} \; z \in \cZ_s, 
                                                                   \; \Stab (z) \; \mbox{is of order 3}
                                                                     \end{array} \right. \right\}.
$$
Note that 
${}_{\mathfrak{S}_3}{\widehat \cM}_2$ is a smooth Deligne-Mumford stack of dimension 3.
The functor $F$ admits a coarse moduli space denoted by
${}_{\mathfrak{S}_3}{\widehat M}_2$. 

According to  \cite{lo2} the stack admits a stratification
$$
{}_{\mathfrak{S}_3}{\widehat \cM}_2 = {}_{\mathfrak{S}_3}{\cM}_2 \sqcup \cR_2 \sqcup \cR_1.
$$
Here $\cR_2$ (respectively $\cR_1$) denotes the locally closed substack of 
${}_{\mathfrak{S}_3}{\widehat \cM}_2$ where  $\cX_s$ admits exactly one node 
(respectively two nodes)
for all $s$.
The index refers to the dimension of the substack. Similarly, there is a stratification  
${}_{\mathfrak{S}_3}{\widehat M}_2 = {}_{\mathfrak{S}_3}{M}_2 \sqcup R_2 \sqcup R_1$ for the corresponding 
moduli spaces.

\subsection{The stacks of genus-2 spin curves}
Recall that a smooth spin curve of genus 2 is a pair $(C,\kappa)$ with $C$ a smooth curve of genus 2 
and $\kappa$ a theta characteristic on $C$, i.e. a line bundle on $C$ whose square is the canonical 
bundle. This definition extends in the obvious way to families of spin curves.
Let $\cS_2$ denote the moduli stack of smooth spin curves of genus 2.
It is associated to the functor $G: Sch/\CC \ra  Ens$ defined by
$$
S \mapsto \left\{\mbox{pairs} \; (\cC \ra S, \kappa_S) \;\left| \begin{array}{l}
                                                               \cC \ra S \; \mbox{is a smooth curve of genus 2 over} \ S,\\
                                                   \kappa_S \; \mbox{ is a theta characteristic of }\; \cC \; \mbox{over} \;S
                                                                     \end{array} \right. \right\}.
$$

The moduli stack $\cS_2$ is a smooth Deligne-Mumford stack which decomposes into two 
irreducible components $\cS_2^+$ and $\cS_2^-$ depending on the parity of the theta characteristic. Thus
a spin curve $(C, \kappa)$  is in $\cS_2^+$  (respectively  $\cS_2^-$) if and only if $h^0(\kappa)$ is an even (respectively odd)
number.  Moreover, the forgetful map  $\cS_2 \ra \cM_2$  onto the 
moduli stack $\cM_2$ of (smooth) curves of genus 2 is of degree 10 over $\cS_2^+$ (respectively 6 over $\cS_2^-$).
The functor $G$ admits a coarse moduli space denoted by $S_2$. 
Clearly we have $S_2 = S_2^+ \sqcup S_2^-$.\\ 

We recall now the compactification $\overline{\cS}_2$ of  $\cS_2$ constructed by Cornalba  in \cite{c} 
(in fact, the construction works for arbitrary genus, we only need the genus-2 case).
A rational component $E \subset X$ of  a nodal curve $X$ is called {\it exceptional} if 
$\# (E \cap \overline{X \setminus E})=2$. The curve $X$ is called {\it quasi-stable} if $\# (E \cap \overline{X \setminus E})\geq2$
for any smooth rational component $E \subset X$ and any two exceptional components are disjoint.  
 A ({\it generalized}) {\it spin curve} of genus 2 is a triple $(X, \kappa, \beta)$, where $X$ is a quasi-stable curve of (arithmetic) genus 2, $\kappa \in \Pic^{1}(X)$ is 
a line bundle such that $\kappa_E= \cO_E(1)$ for every exceptional component $E\subset X$, and $\beta: \kappa^{\otimes 2} \ra \omega_X$ is
a sheaf homomorphism which is generically non-zero along each non-exceptional component of $X$.

When $X$ is smooth, $\kappa$ is an ordinary theta characteristic and $\beta $ is an isomorphism. 
A {\it family of spin curves } over a base scheme $S$ consists of a triple 
$(\cX \stackrel{f}{\ra} S, \eta, \beta)$ where $f : \cX  \ra S$
is a flat family of quasi-stable curves, $\eta \in \Pic(\cX)$ is a line bundle and 
$\beta: \eta^{\otimes 2} \ra \omega_{\cX}$ is a sheaf 
homomorphism, such that at every point $s \in S$ the restriction $(\cX_s, \eta_s, \beta_s)$ is a spin 
curve of genus 2. The compactification
$\overline{\cS}_2$ is the stack  associated to the functor $G: S \mapsto \{ (\cX \stackrel{f}{\ra} S, \eta, \beta) \} $.  Again one has the obvious decomposition
$$
\overline{\cS}_2 = \overline{\cS}_2^+ \sqcup \overline{\cS}_2^-
$$ 
depending on the parity of $h^0(C, \kappa)$.
 
Let us describe the boundary components of $\overline{S}^+_2$ respectively 
$\overline{\cS}^+_2$.
Denote by $\pi: \overline{S}^+_2 \ra 
\overline{M}_2 $ the forgetful map $[X,  \kappa, \beta] \mapsto [C]$, where $C$ is the stable model  
of $X$ is obtained from $X$ by contracting all the exceptional components. 

If $[X, \kappa, \beta ] \in  \pi^{-1}([X_1\cup_y X_2])$ where $X_1$ and $X_2$ are elliptic 
curves intersecting transversally in a point $y$, then necessarily  
$$
X:=X_1\cup_{y_1} E \cup_{y_2}X_2,
$$ 
where $E$ is an exceptional component such that 
$X_1 \cap E = \{y_1\}$ and $X_2 \cap E = \{y_2\}$ with $\pi(y_i) = y$ for $i=1$ and 2. Moreover, 
$$
\kappa= (\kappa_{X_1}, \kappa_{X_2}, \kappa_E= O_E(1)) \in \Pic^1(X),
$$ 
with theta characteristics $\kappa_{X_i}$ on $X_i$ for $i=1$ and 2. 
The condition $h^0(X, \kappa)\equiv 0$ mod $2$ implies that  $\kappa_{X_1}$ and $\kappa_{X_2}$ have the same parity. We denote 
$$
A_1 := \; \mbox{closure of} \left\{(X,\kappa,\beta) \in \overline{S}^+_2 \;\left|\; 
\begin{array}{c}X \; \mbox{and} \;  \kappa \; \mbox{as above with}
\; h^0(\kappa_{X_i}) = 0 \; \mbox{for} \; i = 1,2\\
\mbox{and} \;\beta = \; \mbox{the obvious map (i.e. zero on}\; E)
\end{array} 
\right. \right\}
$$  
and 
$$
B_1 := \; \mbox{closure of} \left\{(X,\kappa,\beta) \in \overline{S}^+_2 \;\left| \; \begin{array}{c} 
X \; \mbox{and} \;  \kappa \; \mbox{as above with}
\; \kappa_{X_i} = \cO_{X_i} \; \mbox{for} \; i = 1,2\\
\mbox{and} \;\beta = \; \mbox{the obvious map (i.e. zero on}\; E)
\end{array}
\right. \right\}.
$$

If $[X, \kappa, \beta]  \in  \pi^{-1}([C])$ with $C$ an irreducible one-nodal curve, 
$\nu: \widetilde C \ra C$ 
denotes the normalization of $C$, and
$$
C = \widetilde C/y_1 \sim y_2
$$ 
where $y_1$ and $y_2$ map to the node $y$ of $C$, then
there are two possibilities, namely
$$
X=C \quad  \mbox{or} \quad  X= \widetilde C\cup_{\{y_1,y_2\}} E
$$ 
with $E$ an exceptional component.
In the first case let $\kappa_{\widetilde C}:= \nu^*(\kappa)$.
Then $ \kappa_{\widetilde C}^{\otimes 2}= K_{\widetilde C}(y_1+y_2)$ and there is only one choice of gluing the fibres 
$\kappa_{\widetilde C}(y_1)$ and $\kappa_{\widetilde C}(y_2)$ to get $\kappa$ such that $h^0(X, \kappa) \equiv 0 $ mod $2$. We denote  
$$
A_0 := \; \mbox{closure of}   \; \left\{(X,\kappa,\beta) \in  \overline{S}^+_2 \; \left| \;
\begin{array}{c}
X =C,  \; \kappa \; \mbox{as above and} \;  \beta \;\mbox{the obvious map} 
\end{array}
\right. \right\}.
$$

If $[X, \kappa, \beta]  \in  \pi^{-1}([C])$ with $X= \widetilde C\cup_{\{y_1,y_2\}} E$, 
then $\kappa_{\widetilde C}:= \kappa \otimes \cO_{\widetilde C}$ is a 
theta characteristic on $\widetilde C$ and $\kappa|_E = \cO_E(1)$.
Since $H^0(X, \omega_X) \simeq H^0(\widetilde C, \omega_{\widetilde C})$, it follows that $ \kappa_{\widetilde C}$ is an even theta characteristic on the elliptic curve $\widetilde C$. We denote 
$$
B_0 := \; \mbox{closure of} \left\{ (X, \kappa, \beta) \in \overline{S}^+_2  \; \left| 
\begin{array}{c} 
X= \widetilde C\cup_{\{y_1,y_2\}} E, \; \kappa \; \mbox{as above}\\
\mbox{and} \; \beta \; \mbox{the obvious map} 
\end{array}
\right. \right\}.
$$
We denote by $\alpha_0, \beta_0, \alpha_1, \beta_1$ the corresponding divisors of the stack $\overline{\cS}_2^+$.
All in all, we have 
$$
\overline{S}^+_2 = S^+_2 \sqcup A_0 \sqcup B_0 \sqcup A_1 \sqcup B_1 \quad \mbox{and}
\quad 
\overline{\cS}^+_2 = \cS^+_2 \sqcup \alpha_0 \sqcup \beta_0 \sqcup \alpha_1 \sqcup \beta_1.
$$

\section{The isomorphism ${}_{\mathfrak{S}_3}{\widehat \cM}_2 \ra \cS_2^+$}

Let $\cJ_2$ be the open substack of the stack of canonically polarized  
Jacobians of smooth curves of genus 2. In \cite{lo2} we showed that the Prym map 
$$
Pr: {}_{\mathfrak{S}_3}{\widehat M}_2 \ra \cJ_2
$$ 
is a finite surjective morphism of degree 10. The description of the fibres of $Pr$ hints that there is a relation between the spaces 
${}_{\mathfrak{S}_3}{\widehat \cM}_2$ and $\overline{\cS}^+_2$. 
Note that the forgetful map $\cS_2^+ \ra \cM_2$ is also of degree 10.
The idea is, to associate to each admissible non-cyclic  $\mathfrak{S}_3$-covering 
$[Z \ra X] \in {}_{\mathfrak{S}_3}{\widehat M}_2$  the theta divisor of the corresponding Prym variety
$P(f)$ of the induced covering $f: Y = Z/\langle \tau \rangle \ra X$, 
which in this case is a smooth genus 2 curve $C$,
together with the theta characteristic arising from a naturally defined 6:1 map  $C \ra  \PP^1$.

\begin{thm} \label{thm3.1}
There is a canonical isomorphism of stacks
$$
\alpha: {}_{\mathfrak{S}_3}{\widehat \cM}_2 \ra \cS_2^+.
$$
\end{thm}

\begin{proof}
Let $[h:\cZ \ra \cX \; \mbox{over}\; S] \in {}_{\mathfrak{S}_3}{\widehat \cM}_2$. So the group $\mathfrak{S}_3 = \langle \sigma, \tau \rangle$ acts on $\cZ$ over $S$. The quotient $\cY := \cZ/\langle \tau \rangle$ is an admissible cover of degree 3 of $\cX$ over $S$. Let 
$$
f: \cY \ra \cX
$$ 
denote the induced morphism. For any closed point $s \in S$ we associated in \cite{lo1} and \cite{lo2} a smooth curve $C(s)$ and a theta
characteristic $\kappa(s)$ on $C(s)$ in the following way: Let $\widetilde{f(s)}: \widetilde{\cY(s)} \ra \widetilde{\cX(s)}$ denote the normalization of the map $f(s): \cY(s) \ra \cX(s)$. The curve 
$\widetilde{\cY(s)} $ is hyperelliptic of genus 4 (respectively 3 or 2 if $\cX(s)$ has 1 or 2 nodes) and there is a commutative diagram (see \cite[proof of Theorem 5.1]{lo1} in the smooth case and 
 \cite[proof of Theorem 9.1]{lo2} in the nodal case)
\begin{equation} \label{commdiag}
\xymatrix{
\widetilde{\cY(s)} \ar[r]  \ar[d]_{\widetilde{f(s)}} & \PP^1 \ar[d]_{\bar{f}(s)} & C(s) \ar[l]_{\varphi(s)} \ar[ld]^{\psi(s)} \\
\widetilde{\cX(s)}  \ar[r] & \PP^1&
}
\end{equation}
where the horizontal maps of the square are the hyperelliptic coverings (respectively a suitable double cover in the nodal case) 
and  $C(s)$ is a smooth curve of genus 2,  which is a theta divisor of the principal polarization of the Prym variety of the covering
$f(s)$,  with hyperelliptic cover $\varphi(s)$. To be more precise, in any case ($\cX(s)$ smooth or not) there are 6 Weierstrass 
points $q_1, \cdots, q_6$ of $\widetilde{\cY(s)}$ and 2 ramification points $p_1$ and $p_2$ of the double cover $\widetilde{\cX(s)} 
\ra \PP^1$ such that 
$$
f(s)(q_1) = f(s)(q_2) = f(s)(q_3) = p_1 \quad \mbox{and} \quad f(s)(q_4) = f(s)(q_5) = f(s)(q_6) = p_2.
$$
We can then consider the Prym variety of $f(s)$ as a subvariety of $\Pic^2(\widetilde{\cY(s)})$ and the symmetric product $\widetilde{\cY(s)}^{(2)}$ as an open subset of $\Pic^2(\widetilde{\cY(s)})$ in the usual way. 
In the smooth case the curve $C(s)$ is given as (see \cite[Section 4.5]{lo1})
$$
C(s) = \{ \cO_Y(y+z) \in \Pic^2(Y) \; | \; f(y) = f(\iota_Yz), z \neq y \} 
$$
with the reduced subscheme structure, where we write $f$ for $f(s)$ and $\iota_Y$ denotes the hyperelliptic involution of $Y$.
In the nodal case there is a slight modification of this definition 
for which we refer to \cite[Section 8]{lo2}. With this notation the curve $C(s)$ in any case irreducible and 
smooth of genus 2. Its Weierstrass points are the 6 points of $C(s)$: 
$$
\omega_1 = [q_1 + q_2],
\omega_2= [q_1 + q_3], \omega_3=  [q_2+q_3], \omega_4=  [q_4+q_5], \omega_5= [q_4+q_6], \omega_6 =[q_5+q_6]
$$ 
(see \cite[Proposition 4.18]{lo1} and \cite[Remark 8.9]{lo2}).
The composed map $\psi(s)$ is the morphism given by the pencil generated by the 2 divisors $2\omega_1 + 2 \omega_2 + 2\omega_3$ and  
$2\omega_4 + 2 \omega_5 + 2\omega_6$. Clearly, the linearly equivalent divisors 
$\omega_1 +\omega_2 - \omega_3$ and $\omega_4 + \omega_5 - \omega_6$ of $C(s)$ define an even theta characteristic $\kappa(s)$ on $C(s)$ which is uniquely and canonically determined by the covering $f(s)$.

By construction the pair $(C(s), \kappa(s))$ depends algebraically on $s$. Hence we get a family of pairs 
$$
\alpha(h:\cZ \ra \cX):= (\cC, \kappa) \in \cS_2^+
$$ 
over $S$ with $\cC_s = C(s)$ and similarly for $\kappa$. Clearly the family is flat as a family of smooth curves.

Conversely, let $[(\cC,\kappa)$ over $S] \in \cS_2^+$. So for every closed point $s \in S$, 
the curve $C(s) :=\cC_s$ is a smooth curve of genus 2 and $\kappa(s) := \kappa_s$ an even theta characteristic on it. Recall that if $W = \{\omega_1, \dots, \omega_6 \}$ are the Weierstrass points of $C(s)$, the even theta characteristics are of the form $\kappa(s) = \cO_{C(s)}(\omega_i + \omega_j - \omega_k) \simeq  \cO_{C(s)}(\omega_l + \omega_m - \omega_n)$ where $\omega_i,\omega_j,\omega_k$ are different points of $W$ and  $\omega_l,\omega_m,\omega_n$ their complement in $W$. Hence $\kappa(s)$ determines 
a partition of $W$ into 2 complemetary subsets of 3 elements of $W$. We use the associated divisors $2\omega_i + 2 \omega_j + 2\omega_k$ and $2\omega_l+ 2 \omega_m+ 2\omega_n$ to define a degree 6 covering $\psi(s): C(s) \ra \PP^1$ which certainly factorizes via the hyperelliptic covering $\varphi(s): C(s) \ra \PP^1$ and a degree 3 covering ${\overline f}(s): \PP^1 \ra \PP^1$.
So we get the right hand triangle of diagram \eqref{commdiag}.

The map  ${\overline f}(s)$ is certainly unramified at the 6 points ${\overline \omega_i} = \varphi(\omega_i)$ and maps ${\overline \omega}_i, {\overline \omega}_j$ and 
${\overline \omega}_k$ to ${\overline p}_1$ and ${\overline \omega}_l, {\overline \omega}_m$ and 
${\overline \omega}_n$ to ${\overline p}_2$ say. According to the Hurwitz formula, ${\overline f}(s)$ is of ramification degree 4. Suppose first that ${\overline f}(s)$ is simply ramified. Let ${\overline \omega}_7, \dots,
{\overline \omega}_{10}$ denote the ramification points and ${\overline p}_3, \dots, 
{\overline p}_6$ the corresponding branch points. Then define hyperelliptic curves $\cY(s)$ 
with ramification points over ${\overline \omega}_1, \dots, {\overline \omega}_{10}$ and $\cX(s)$ 
with ramification points over ${\overline p}_1, \dots, {\overline p}_6$. 

Now suppose ${\overline f}(s)$ is simply ramified at ${\overline \omega}_7$ and ${\overline \omega}_8$ and doubly ramified at 
${\overline \omega}_9$ and let ${\overline p}_3, {\overline p}_4$ as well as ${\overline p}_5$ the corresponding 
branch points.  In this case, define hyperelliptic curves $\cY(s)$ 
with ramification points over ${\overline \omega}_1, \dots, {\overline \omega}_{8}$ and a
node over ${\overline \omega}_{9}$ and a curve $\cX(s)$ of genus 1 
with ramification points over ${\overline p}_1, \dots, {\overline p}_4$ and a
node over ${\overline p}_5$.

Finally, suppose ${\overline f}(s)$ is doubly ramified at ${\overline \omega}_7$ and 
${\overline \omega}_8$ and let ${\overline p}_3, {\overline p}_4$ denote the corresponding 
branch points. Then define hyperelliptic curves $\cY(s)$ 
with ramification points over ${\overline \omega}_1, \dots, {\overline \omega}_{6}$ and 
nodes over ${\overline \omega}_{7}$ and ${\overline \omega}_8$ and a curve $\cX(s)$ of genus 0 
with ramification points over ${\overline p}_1$ and  ${\overline p}_2$ and nodes 
over ${\overline p}_3$ and ${\overline p}_4$. 

Looking at the ramification one immediately checks that in any case the map ${\overline f}(s)$
lifts to an admissible covering $f(s): \cY(s) \ra \cX(s)$. So we obtain a commutative diagram, 
whose normalization is diagram \eqref{commdiag}. Moreover, it is clear from the constuction that
the diagram varies algebraically with $s \in S$. So we get a family $f: \cY \ra \cX$ over $S$. Since 
in any case the coverings $f_s = f(s)$ are not Galois, the Galois closure $h: \cZ \ra \cX$ of $f: \cY \ra \cX$ 
 over $S$ is an admissible $\mathfrak{S}_3$-cover over $S$.
Clearly we obtain an element  
$$
\beta((\cC, \kappa)) :=  [h: \cZ \ra \cX \; \mbox{over} \; S] \in {}_{\mathfrak{S}_3}{\widehat \cM}_2
$$ 
over $S$. It is easy to verify that $\beta$ is inverse to $\alpha$.

It remains to check that isomorphisms in the category ${}_{\mathfrak{S}_3}{\widehat \cM}_2$ are
mapped by $\alpha$ to isomorphisms in the category $\cS_2^+$ and conversely. This 
follows from the fact that the maps $\alpha$ and $\beta$ are 
determined completely by the Weierstrass points and the
 ramifications of the involved coverings and under isomorphisms these are mapped to 
Weierstrass points and ramifications of the same type.  
 \end{proof}

As mentioned above, both stacks ${}_{\mathfrak{S}_3}{\widehat \cM}_2$ and $\cS_2^+$ admit coarse moduli
spaces ${}_{\mathfrak{S}_3}{\widehat M}_2$ and $S_2^+$. Since an isomorphism of stacks induces an isomorphism of 
the associated coarse moduli spaces, we obtain as an immediate consequence the following corollary.  

\begin{cor}
The isomorphism $\alpha: {}_{\mathfrak{S}_3}{\widehat \cM}_2 \ra \cS_2^+$ induces an isomorphism of  
moduli spaces ${}_{\mathfrak{S}_3}{\widehat M}_2 \ra S_2^+$, also denoted by $\alpha$.
\end{cor}

\section{Non-extension of the construction defining  $\alpha$}

In \cite{lo2} we considered the following subscheme of the moduli scheme 
${}_{\mathfrak{S}_3}{\overline M}_2$ of admissible 
$\mathfrak{S}_3$-covers of stable curves of genus 2:
$$
 {}_{\mathfrak{S}_3}{\widetilde M}_2 := \left\{ [h:Z \ra X] \in {}_{\mathfrak{S}_3}{\overline M}_2 \; \left| \begin{array}{c}
                                                                     p_a(Z) = 7 \; \mbox{and for any node} \; z \in Z \\
                                                                     \mbox{the stabilizer}\; \Stab (z) \; \mbox{is of order 3}

                                                                  \end{array} \right. \right\}               
$$
and we showed that
$$
 {}_{\mathfrak{S}_3}{\widetilde M}_2 =  {}_{\mathfrak{S}_3}{\widehat M}_2 \sqcup S
$$
where $S$ is the closed subscheme of $ {}_{\mathfrak{S}_3}{\widetilde M}_2$ consisting of reducible 
$\mathfrak{S}_3$-coverings.
In this section we study the question whether the isomorphism 
$\alpha: {}_{\mathfrak{S}_3}{\widehat M}_2 \ra S_2^+$ extends to a morphism 
${}_{\mathfrak{S}_3}{\widetilde M}_2 \ra \overline{S}_2^+$.

\begin{lem}\label{extension}
There is a non-empty open set $U \subset S$ such that $\alpha$ extends to a holomorphic map
${}_{\mathfrak{S}_3}{\widehat M}_2 \sqcup U \ra \overline{S}_2^+$.
\end{lem}

Note that ${}_{\mathfrak{S}_3}{\widehat M}_2 \sqcup U$ is open in 
${}_{\mathfrak{S}_3}{\widetilde M}_2$ with complement of codimension $\geq 2$.

\begin{proof}
The moduli space ${}_{\mathfrak{S}_3}{\overline M}_2$ of admissible $G$-covers of stable curves 
of genus 2 is a normal projective variety (see \cite{acg}). So the open subvariety 
${}_{\mathfrak{S}_3}{\widetilde M}_2 $ is also normal. Hence its singular locus is of codimension 
$\geq 2$. Now $\alpha$ defines a rational map $\overline {\alpha}$ from the smooth variety 
$({}_{\mathfrak{S}_3}{\widetilde M}_2)_{reg}$  to $\overline{S}_2^+$. Since any such rational 
map extends to a holomorphic map in codimension 1 (see \cite[p.491]{gh}), this implies the 
assertion. 
\end{proof}

However we have,

\begin{prop} \label{prop4.1}
The construction giving the isomorphism $\alpha: {}_{\mathfrak{S}_3}{\widehat M}_2 \ra S_2^+$ does not 
extend to ${}_{\mathfrak{S}_3}{\widetilde M}_2$.
\end{prop}

In order to see what we have to show, let us recall the definition of the map $\alpha$ in the smooth 
case. Let $Z \ra X$ be an \'etale $\mathfrak{S}_3$-covering of a smooth curve $X$ of genus 2 and 
$f: Y \ra X$ a corresponding non-cyclic \'etale degree-3 covering. Let $h_X$ denote the 
hyperelliptic line bundle on $X$ and $\Theta$ the canonical theta divisor in $\Pic^3(Y)$ given by the image of the map
$Y^{(3)} \ra \Pic^3(Y)$. If $q$ is one of the 2 Weierstrass points of $Y$ such that $f^{-1}(f(q))$ 
consists of 3 Weierstrass points, we consider the following translate of $\Theta$:
$$
\Theta_q := \Theta -q \subset \Pic^2(Y).
$$
Let $\Nm: \Pic^2(Y) \ra \Pic^2(X)$ be the norm map. We define $\widetilde P$
as the following translate of the Prym variety $P$ of $f$:
$$
\widetilde P := \Nm^{-1}(h_X) \subset \Pic^2(Y).
$$ 
Then we have (\cite[Corollary 4.12]{lo1}) 
$$
\widetilde P \cap \Theta_q = \Xi_1 \cup \Xi_2 \cup \Xi_3,
$$
where the $\Xi_i$ are algebraically equivalent divisors, each of them defining the principal polarization of 
$\widetilde P$ and thus pairwise isomorphic smooth curves of genus 2. If $\alpha(f) = (C, \kappa)$,
by definition one has 
$$
C = \Xi_1.
$$
For the proof of Proposition \ref{prop4.1} the theta characteristic $\kappa$ is not relevant. \\

Now,  let $f:Y \ra X$ be the non-cyclic degree-3 covering given by a general element of $S$. So
$ Y = Y_1 \cup_{y_0} Y_2$ and $X = X_1 \cup_{x_0} X_2$ with smooth curves $Y_i$ of genus 2 
(respectively $X_i$ of genus 1) intersecting transversally in the point $y_0$ (respectively $x_0$)
and $f = f_1 \cup_{y_0} f_2$ with non-cyclic degree-3 covers $f_i: Y_i  \ra X_i$ totally ramified in the 
point $y_{0i} \;(=y_0)  \in Y_i$ for $i=1,2$ (we denote the point $y_0$ by $y_{0i}$ when considered as a point of $Y_i$). In order to 
prove Proposition \ref{prop4.1}, we will show that, defining $\widetilde P$ and
$\Theta_q$ in the same way as in the smooth case, the intersection $\widetilde P \cap \Theta_q$ is not 
proper, i.e. not a divisor in $\widetilde P$. We need first some preliminaries.

Let $q_i^0, \dots, q^5_{i}$  denote the Weierstrass points of $Y_i$ for $i = 1,2$. Then $y_0$ is 
necessarily 
one of these points, say $y_0 = q^0_1 = q^0_2$. For $i=1,2$, the hyperelliptic involution of $X_i$
lifts to an involution on $Y_i$, which induce an involution on $Y$. Thus we get the following commutative diagram

\begin{equation} \label{diag3}
\xymatrix@R=25pt@C=40pt {
        Y = Y_1 \cup_{y_0} Y_2\ar[r]^{\gamma = \gamma_1 \cup_{y_0} \gamma_2}_{2:1} 
\ar[d]^{3:1}_{f = f_1 \cup_{x_0} f_2} & 
\PP^1  \cup_{\gamma(y_0)} \PP^1 \ar[d]_{3:1}^{\overline{f} = \overline{f}_1 
\cup_{\gamma(y_0)} 
\overline{f}_2}  \\
       X = X_1 \cup_{x_0} X_2  \ar[r]^{\delta = \delta_1 \cup_{x_0} \delta_2}_{2:1}  &  \PP^1 
\cup_{\delta(x_0)} \PP^1
    }
    \end{equation}
where $\gamma_i$ and $\delta_i$ are the hyperelliptic coverings.
It follows that the image of any Weierstrass point under the map $f_i$ is a ramification point of 
$\delta_i$. Since $\delta_i$ admits 4 ramification points, say $p^0_i, \dots, p^3_i$, we conclude 
that there is one point $p^j_i$ such that $f_i^{-1}(p^j_i)$ consists of 3 Weierstrass points and 3 points $p^j_i$ such that  $f^{-1}(p^j_i)$
contains only one Weierstrass point. Without loss of generality we may assume that 
$$
f(q^1_i )=f(q^2_i)=f(q^3_i)= p^1_i
$$ 
and 
$$
 f(q^4_i) = p^2_i, \quad f(q^5_i)=p^3_i \quad \mbox{and } \quad  f(q^0_i)=p^0_i = x_{0i}.
$$ 

The normalization of $Y$ (respectively $X$) is given by $n_Y = \iota_{Y_1} \cup \iota_{Y_2}$ 
(respectively $n_X = \iota_{X_1} \cup \iota_{X_2}$) where $\iota_{Y_i}$ (respectively $\iota_{X_i}$) 
denote the canonical embeddings $Y_i \ra Y$ (respectively $X_i \ra X$). They induce canonical isomorphisms of the Picard varieties
\begin{equation}  \label{eq4.2}
n_Y^*: \Pic(Y) \ra \Pic(Y_1) \times \Pic(Y_2) \qquad \mbox{and} \qquad n_X^*: \Pic(X) \ra 
\Pic(X_1) \times \Pic(X_2).
\end{equation} 
In the sequel we identify both sides, i.e. denote the elements of $\Pic(Y)$ and $\Pic(X)$ by pairs 
$(L_1,L_2)$ with $L_i \in \Pic(Y_i)$ (respectively $\Pic(X_i)$).

Now $\Pic^3(Y)$ consists of infinitely many components, however there are only 2 components
namely $\Pic^{(2,1)}(Y)$ and $\Pic^{(1,2)}(Y)$ (with the obvious notation) which admit a canonical
theta divisor (see \cite[Proposition 2.2]{b}; these are also the only balanced components in the sense 
of Caporaso's compactified Picard varieties \cite{c}). Since the situation is symmetric in $Y_1$ 
and $Y_2$, we work only with $\Pic^{(2,1)}(Y)$. Then the canonical theta divisor is
$$
\Theta : = \{ L \in \Pic^{(2,1)}(Y) = \Pic^2(Y_1) \times \Pic^1(Y_2) \;| \; h^0(L) \geq 1 \}
$$
with reduced subscheme structure. The following lemma is shown by using Riemann-Roch formula. 

\begin{lem} \label{lem4.2}
With the identifications $Y_1 + y_{01} = \{ \cO_{Y_1}(y_1 + y_{01}) \;|\; y_1 \in Y_1 \}$ and 
$Y_2 = \{\cO_{Y_2}(y_2) \;|\; y_2 \in Y_2 \}$ we have
$$
\Theta = \left[(Y_1 + y_{01}) \times \Pic^1(Y_2) \right] \cup \left[ \Pic^2(Y_1) \times Y_2 \right].
$$
\end{lem}

If a family of smooth $\mathfrak{S}_3$-coverings degenerates to the map $f:Y \ra X$ described above, 
the Weierstrass point $q$ chosen for the translate of $\Theta$, specializes to one of the points $q^i_1$ or $q^i_2$ 
with $i = 1,2$ or 3. In the latter case $\Theta - q^i_2$ is a divisor of $\Pic^{(2,0)}$ which is not balanced. So we need to work with
one of the $q^i_1$, say $q = q^1_1$ and define
$$
\Theta_q := \Theta - q \in \Pic^{(1,1)}(Y).
$$
There is exactly one line bundle of degree 2 with $h^0 = 2$ in $\Pic^{(1,1)}(Y)$, respectively 
$\Pic^{(1,1)}(X)$, namely 
$$
h_Y = (\cO_{Y_1}(q_{01}),\cO_{Y_2}(q_{02})),  \quad \mbox{respectively} \quad
h_X = (\cO_{X_1}(p_{01}),\cO_{X_2}(p_{02})).
$$

If $\Nm_f: \Pic^2(Y) \ra \Pic^2(X)$ denotes the norm map we define, as in the smooth case,
$\widetilde P$ as the following translate of the Prym variety $P$ of $f$:
$$
\widetilde P := \Nm_f^{-1}(h_X) \subset \Pic^{(1,1)}(Y).
$$ 

\begin{lem} \label{lem4.3}
The intersection $\Theta_q \cap \widetilde P$ contains the $3$ pairwise disjoint curves 
$$
\Sigma_i := \{ (\cO_{Y_1}(q^i_1 + y_{01} - q), \cO_{Y_2}(y_{02} + y_2 - y'_2)) \; | \;
y_2 \in Y_2, \;y'_2 \in f_2^{-1}f_2(y_2) \}
$$
for $i = 1,2,3$ which are $3:1$-coverings of $Y_2$.
\end{lem}

\begin{proof}
According to Lemma \ref{lem4.2} we have $\Sigma_i \subset \Theta_q$. By  the identifications
\eqref{eq4.2}, $\Nm_f = \Nm_{f_1} \times \Nm_{f_2}$ and one computes 
$$
\hspace{-3cm} \Nm_f((\cO_{Y_1}(q^i_1 + y_{01} - q), \cO_{Y_2}(y_{02} + y_2 - \iota_{Y_2}(y_2)) = {}{}{}{}
$$
$$ 
\hspace{2cm} = (\cO_{X_1}(p_1 + x_{01} - p_1), \cO_{X_2}(x_{02}  + f_2(y_2) - f_2(y_2))
$$
$$
\hspace{-1.2cm} = (\cO_{X_1}(x_{01}),\cO_{X_2}(x_{02})) = h_X.
$$
So $ \Sigma_i \subset \widetilde P$. The curves $\Sigma_i$ are pairwise disjoint, since $q = q^1_1$ and the line bundles 
$\cO_{X_1}(q^i_1 + y_{01} - q^1_1)$ are pairwise different. The last assertion follows from the fact that $f_2$ is a 
$3:1$-covering and hence there are 3 preimages $y_2' \in f_2^{-1}f_2(y_2)$.
\end{proof}

\begin{proof}[Proof of Proposition 4.2]
We may consider $\Pic^{(1,1)}(Y)$ as an abelian variety and
$\widetilde P$ as an abelian subvariety, since both contain the distinguished point $h_Y$. 
According to \cite[Proposition 5.3]{lo2},
$$
\widetilde P = P_1 \times P_2
$$
with elliptic curves $P_i = \Prym(f_i)$ for $i=1, 2$ and canonical principal polarization.
On the other hand, the principal polarization of $\Pic^{(1,1)}(Y)$ is defined by the divisor $\Theta_q$. 
So if the construction of $\alpha$ would extend to the closed subvariety $S$, the divisor $\Theta_q$ would restrict to a divisor 
defining the threefold of the canonical principal polarization of $\widetilde P$. Being pairwise disjoint, the curves $\Sigma_i$ 
would define the canonical principal polarization i.e. would be isomorphic to $P_1 \times \{0\} \cup \{0\} \times P_2$.
But this contradicts Lemma \ref{lem4.3} since, the curve $Y_2$ being of genus 2, any 3:1 
covering is of arithmetic genus $>2$ and hence $\Sigma_i$ cannot define a principal polarization. 
\end{proof}

\section{The nef cone of $\overline{S}_2^+$}

In this section we compute the cone of numerically effective divisors of $\overline{S}_2^+$
in the rational Picard group $\Pic_{\QQ}(\overline{S}_2^+)$.
For this we consider the moduli space $\overline{M}_{0,[3,3]}$ of stable curves of genus 0 with 
6 unordered marked points partitioned into 2 sets of 3 points. If $\overline {M}_{0,6}$ denotes the 
usual moduli space of stable 6-pointed curves of genus 2 and 
$$
G := (\mathfrak{S}_3 \times \mathfrak{S}_3) \rtimes \langle \tau \rangle
$$
where the first $\mathfrak{S}_3$ acts on the numbers $1,2,3$, the second on the numbers
$4,5,6$ and $\tau= (14)(25)(36)$, then  
$$
\overline{M}_{0,[3,3]} = \overline{M}_{0,6} / G.
$$
Here we consider the numbers as indices of the marked points.
We use the fact (see \cite[Lemma 20]{k}) that there is a canonical isomorphism
$$
\overline{M}_{0,[3,3]} \simeq \overline{S}_2^+.
$$
In the sequel we often identify both spaces and denote corresponding divisors by the same 
letter. Moreover, the composed map $\overline {M}_{0,6} \ra \overline{S}_2^+$ maps the boundary of 
$\overline {M}_{0,6}$ onto the boundary of  $\overline{S}_2^+$ (\cite{k}).

Recall that the boundary of  $\overline{M}_{0,6}$ consists of divisors $\Delta_S$, where 
$S\subset \{ 1,\ldots, 6\}$ with $|S|, |S^c| \geq 2$. Each $\Delta_S$ is the closure of the points
corresponding to reducible curves with one node and where $S$ points are marked in one 
component. We denote the class of $\Delta_S$ in $\Pic_{\QQ}(\overline{M}_{0,6})$ by the same letter.

The divisors in  $\overline{M}_{0,[3,3]}$ can be regarded as the divisors
in $\overline{M}_{0,6}$ which are invariant under the action of $G$.
The isomorphism $\overline{M}_{0,[3,3]} \simeq \overline{S}_2^+$ is defined in a natural way 
 by associating to a rational curve with six marked points the admissible double 
covering ramified over those points (\cite{k}). Moreover, it  maps the following $G$-invariant divisors into the
boundary divisors of  $\overline{S}_2^+$(see \cite[table in section 3.1]{k}) :

\begin{eqnarray} \label{eq3.1}
\Delta^{11}:= \sum_{S \in Orb_G (12)} \Delta_S & \mapsto & A_0\\ \label{eq3.2}
 \Delta^{12}:=  \sum_{S \in Orb_G (14)} \Delta_S & \mapsto & B_0\\ \label{eq3.3}
 \Delta_{123}^{c}:=  \sum_{S \in Orb_G (124)} \Delta_S & \mapsto & A_1\\ \label{eq3.4}
    \Delta_{123} := \Delta_{\{1,2,3\}}  & \mapsto & B_1
\end{eqnarray}

In order to see the $G$-invariance of these divisors, use the fact  that if $\{i_1, \dots, i_6\} 
= \{1, \dots,6\}$,
then $\Delta_{\{i_1,i_2,i_3\}} = \Delta_{\{i_4,i_5,i_6\}}$.
When we consider the $G$-invariant classes 
$\Delta^{11}, \Delta^{12}$, $\Delta_{123}^{c}$ and $\Delta_{123}$
as elements of $\Pic_{\QQ}(\overline{M}_{0,[3,3]})$, we denote them respectively  by $A_0,B_0, A_1$ and $B_1$. 
According to \cite{k} the classes $A_0,B_0,A_1$ and $B_1$
generate the $\QQ$-vector space $\Pic_{\QQ}(\overline{M}_{0,[3,3]})$.

In order to compute the nef cone of $\overline{M}_{0,[3,3]}$ we introduce some 
$F$-curves. An {\it $F$-curve in} $\overline{M}_{0,6}$ is obtained as the image of the map $\nu: \overline{M}_{0,4} 
\ra \overline{M}_{0,6}$ defined by attaching one 3-pointed curve at one of the 4 
marked points or two 2-pointed curves  at two of the marked points  (\cite[end of Section 2]{fg}). Note that, as a complete 
intersection curve in $\overline{M}_{0,6}$, the curve $F$ always consists of 3 irreducible components, in the first case 
the attached curve is semistable with 2 irreducible components, whereas in the the second case the 2 attached curves 
are irreducible.

We consider the following $F$-curves. Let $\Gamma_1$ (respectively $\Gamma_2$) denote the 
curve in $\overline{M}_{0,6}$ whose elements are constructed by attaching to a spine labelled 
with the subset $\{1,2,3\}$ (respectively $\{1,2,4\}$) of  $\{1,\dots ,6\}$ to a 4th point a rational
curve consisting of 2 components where the component directly attached to the spine is labelled with 
4 (respectively 3) and the other curve with 5 and 6 (see Figure 1). 
This defines a curve in $\overline{M}_{0,6}$, since the 4th point of the spine moves freely, whereas the other components 
have 3 marked points: the nodes and the labelled points.

Let $\Gamma_3$ be the curve constructed by attaching 
two 2-pointed irreducible curves to the spine where exactly one of the attaching points moves and the 
2 points on the spine are labelled with $\{1,2\}$.
Similarly, we define $\Gamma_4$ (respectively $\Gamma_5$)  by attaching two 2-pointed 
curves and where the labelled points on the spine are $\{1,4\}$ and one tail is labelled 
with $\{2,3\}$ (respectively $\{2,5\}$) (see Figure 1). 

\begin{figure}[h!]
\begin{center}
\includegraphics[width=10cm]{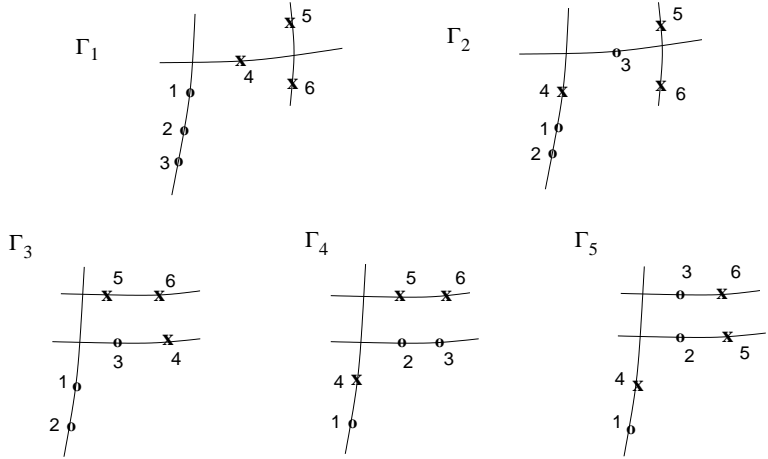}
\caption{Generic points of the F-curves}
 \end{center}
\end{figure}

In principle one should consider the orbit of an $F$-curve under the action of the group $G$, but 
in order to determine the inequalities defining the nef-cone it will be sufficient to intersect the 
divisors with a representative of the orbit.

\begin{lem} \label{lem3.2}
A rational divisor class $aA_0 + b B_0 + c A_1 + d B_1$ of  $\overline{S}_2^+$ 
is nef (respectively ample) 
if and only if the corresponding class 
$D := a\Delta^{11} + b \Delta^{12} + c \Delta_{123}^c + d \Delta_{123}$ 
of $\overline {M}_{0,6}$ satisfies
$$ 
(D \cdot \Gamma_i) \geq 0 \quad (respectively >0)
$$
for $i = 1,\dots, 5$.
\end{lem}

\begin{proof}
The class of $D$ being $G$-invariant, $D$ is nef as a divisor of $\overline{M}_{0,[3,3]}$ if and 
only if it is nef as a divisor of $\overline{M}_{0,6}$. Hence according to \cite[Theorem 1.3]{kmc}
$D$ is nef if and only if it intersects any $F$-curve of $\overline{M}_{0,6}$ non-negatively.
Since $D$ is $G$-invariant, it suffices to check this for a representative of the $G$-orbits of 
$F$-curves. Now it is easy to check that there are exactly 6 $G$-orbits with representatives $\Gamma_1,
\dots, \Gamma_5$ and $\Gamma'_2$. Here $\Gamma'_2$ differs from $\Gamma_2$ only
by labelling the middle component with 5 (instead of 3) and the last component with 3 and 6
(instead of 5 and 6). However, by 
\cite[Lemma 4.3]{kmc} $\Gamma'_2$ is numerically equivalent to $\Gamma_2$. So this implies the assertion on the nefness and also on the ampleness of the divisor class.
\end{proof}

\begin{lem} \label{lem3.3}
The intersection numbers of the divisors  $\Delta^{11}, \dots, \Delta_{123}$ with $\Gamma_i$ are
$$
(\Delta^{11} \cdot \Gamma_1) =  3 , \quad  (\Delta^{12} \cdot \Gamma_1) = 0, \quad 
(\Delta_{123}^c \cdot \Gamma_1) = 0 , \quad (\Delta_{123} \cdot \Gamma_1) =   -1 ;
$$
$$
(\Delta^{11} \cdot \Gamma_2) =   1, \quad  (\Delta^{12} \cdot \Gamma_2) = 2,\quad 
(\Delta_{123}^c \cdot \Gamma_2) =  -1 , \quad 
(\Delta_{123} \cdot \Gamma_2) =   0 ;
$$
$$
(\Delta^{11} \cdot \Gamma_3) =  0 , \quad  (\Delta^{12} \cdot \Gamma_3) = -1,\quad 
(\Delta_{123}^c \cdot \Gamma_3) = 2 , \quad 
(\Delta_{123} \cdot \Gamma_3) =   0 ;
$$
$$
(\Delta^{11} \cdot \Gamma_4) =   -2, \quad  (\Delta^{12} \cdot \Gamma_4) = 1,\quad 
(\Delta_{123}^c \cdot \Gamma_4) =  1 , \quad 
(\Delta_{123} \cdot \Gamma_4) =   1;
$$
$$
(\Delta^{11} \cdot \Gamma_5) =  0 , \quad  (\Delta^{12} \cdot \Gamma_5) = -1,\quad 
(\Delta_{123}^c \cdot \Gamma_5) =   2, \quad 
(\Delta_{123} \cdot \Gamma_5) =   0;
$$
\end{lem}

\begin{proof}
This follows by an immediate computation using \cite[Lemma 4.3]{kmc} which gives the intersection numbers of any boundary divisor with any $F$-curve of $\overline{M}_{0,n}$. 
\end{proof}

From this we immediately conclude the following criterion for a rational divisor class on
$\overline{M}_{0,[3,3]} = \overline{S}_2^+$ to be nef or ample. 

\begin{prop} \label{prop3.4}
A $\QQ$-divisor $D \equiv aA_0 + + bB_0 + cA_1 + d B_1$ is nef (respectively ample)
if and only if the following inequalities are satisfied 
$$
3a \geq d,  \quad a + 2b  \geq c, \quad 2c \geq b \quad \mbox{and} \quad b+c+d \geq 2a 
$$
(respectively all inequalities are strict).
\end{prop}
\begin{proof}
According to Lemma \ref{lem3.3} and the identifications \eqref{eq3.1}, $\dots$, \eqref{eq3.4}  
we have
$$
(D \cdot \Gamma_1) =3a - d, \quad ( D \cdot \Gamma_2) = a + 2b -c,  
$$
$$
(D \cdot \Gamma_3) = (\Gamma \cdot \Gamma_5) = -b + 2c, \quad (D \cdot \Gamma_4) = -2a + b + c + d.
$$

So Lemma \ref{lem3.2} implies the assertion.
\end{proof}

\begin{rem} According to \cite{k} the vector space 
$\Pic_{\QQ}(\overline{S}_2^+) = \Pic_{\QQ}(\overline{M}_{0,[3,3]})$ is of dimension 3. Hence there is a non-trivial relation
$$
a A_0 + b B_0 + c A_1 + d B_1 = 0
$$ 
between the classes $A_i$ and $B_i$. Intersecting with the curves $\Gamma_i$ we get the system of equations
$$
3a = d, \quad a + 2b = c, \quad 2c = b, \quad b+c+d = 2a.
$$
This gives the relation
\begin{equation}  \label{eq4.5}
3A_0 -2B_0 - A_1 + 9 B_1 = 0.
\end{equation} 
in $\Pic_{\QQ}(\overline{S}_2^+)$. Note that this is different from the relation proved in \cite[Lemma 21 (ii)]{k}.
\end{rem}

Since  $\Pic_{\QQ}(\overline{S}_2^+)$ is of dimension 3, \eqref{eq4.5} implies that 
$\{ A_0, B_0, B_1\}$ is a basis of $\Pic_{\QQ}(\overline{S}_2^+)$. In terms of this basis Proposition \ref{prop3.4} 
can be rephrased as

\begin{cor} \label{cor4.5}
A divisor $D = aA_0 + bB_0 + cB_1$ of $\overline{S}_2^+$ is nef (respectively ample) if and only if 
$$
 b \leq 0 \quad \mbox{and} \quad \max (\frac{1}{3}c, -2b) \leq a \leq \frac{1}{2}(b+c).
$$
\end{cor}

\section{Proof of part (1) of Theorem 2}

As in Section 2 let $\overline{\cS}_2^+$ and $\overline{\cM}_2$  denote the corresponding moduli stacks of  $\overline{S}_2^+$ and
$\overline{M}_2$.
Consider the following commutating diagram:
$$
\xymatrix{
        \overline{\cS}_2^+  \ar[r]^{\tilde{q}}  \ar[d]_{\tilde{\pi}}  & \overline{S}_2^+ \ar[d]^{\pi} \\
         \overline{\cM}_2 \ar[r]^{q}  & \overline{M}_2 
    }
 $$
where $\pi$ (resp. $\tilde{\pi}$) is the forgetful map $[(C, \kappa)] \mapsto [C]$ and  $\tilde{q}$ and
$q$ are the natural maps. The first aim is to give a formula for the class $K_{\overline{\cS}_2^+}$ in terms of
divisors in $\overline{S}_2^+$. 

Let $E \subset  M_2 $ the locus of the 
bielliptic curves. Any bielliptic curve $C \in E$ has an equation of the form 
\begin{equation}  \label{eq6.1}
y^2 = (x^2-x_1^2)(x^2-x_2^2)(x^2-x_3^2).
\end{equation} 
Thus the Weierstrass points of $C$ are $\pm x_1, \pm x_2, \pm x_3$. Let $\rho$ denote the bielliptic involution in $C$.

\begin{lem} \label{biellocus}

The preimage of the bielliptic locus $E$ in $\overline{S}_2^+$ decomposes as
$$
\pi^*(E) = \widetilde {E} \cup \widetilde {E}'
$$
where $\widetilde {E}$ represents the spin curves $(C,\kappa)$ with $C$ admitting a bielliptic 
involution  $\rho$ such that $\rho^*(\kappa) = \kappa$. 
Moreover the map $\pi: \widetilde {E} \mapsto E $  is finite of degree $4$ . In particular 
$\widetilde {E}$ is the bielliptic locus of $\overline{S}_2^+$.
\end{lem}

\begin{proof}
Given a bielliptic curve $[C] \in E$  with equation \eqref{eq6.1} and Weierstrass points 
$\pm x_1, \pm x_2, \pm x_3$, it suffices  to check that 
there are exactly 4 even theta characteristics on $C$ which are invariant under the action of 
$\rho$.

We set  $x_i':= -x_i $ for $i=1,2,3$. The involution $\rho$ acts on the Weierstrass points by $x_i \mapsto x_i' $.  
The  even theta characteristics on $C$ are of the form 
$$
p_1+p_2 - p_3 \sim p_4+p_5-p_6
$$
where $ \{ p_1 , \dots , p_6 \}$  is the set of Weierstrass points. Using this equivalence is easy to verify that 
exactly the following theta characteristics are fixed by $\rho$:
 $$
x_1 + x_2 - x_3 ,   \quad  x_1+x_2 - x_3', \quad  x_1+ x_2' - x_3   \quad  x_1'+x_2-x_3.
$$
Thus the action of $\rho$ decomposes $\pi^*E$ into to divisors: $\widetilde {E}$ where $\rho$
fixes the theta characteristics and $\widetilde {E}'$ where $\rho$ permutes non-trivially  the 
remaining 6 theta 
characteristics, which are
$$ 
   x_i+x_i' - x_j, \quad  1\leq i,j \leq 3, \ i\neq j.
$$
\end{proof}

Recall that $\overline{M}_2 \setminus M_2 = \Delta_0 \cup \Delta_1$, where $\Delta_0$  (respectively 
$\Delta_1$) is the closure of the locus of irreducible nodal curves  (respectively 
of the reducible curves) of genus 2.

\begin{lem} \label{ram}
The forgetful map  $\pi: \overline{S}_2^+ \ra \overline{M}_2 $ is simply ramified on $B_0$
and unramified everywhere else in codimension $1$.
\end{lem}

\begin{proof}
Clearly $\pi$ is \'etale of degree 10 over the smooth locus of $M_2$. We have 
$\pi^{-1}(\Delta_0) =A_0 \cup B_0$.
From the construction of the divisor $A_0$ we obtain that $\deg (\pi|_{A_0}) = 4$ (since an 
elliptic curve admits exactly 4 theta characteristics). 
Over a semistable curve $\tilde{C} \cup_{y_1,y_2} R$
with $g(\tilde{C})=1$ and $R$  a rational component
there are exactly 3 even theta characteristics (since an elliptic curve admits exactly 3 even theta characteristics). 
From $\deg (\pi|_ {B_0}) = 6$ we deduce that $\pi$ is simply ramified in $B_0$. One verifies that $\pi$ is \'etale over
$\Delta_1$ since $\pi^{-1}(\Delta_1)= A_1 \cup B_1$,  $\deg (\pi|_{A_1}) =9 $  and $\pi$ maps $B_1$ isomorphically
onto $\Delta_1$.
\end{proof}

We use the previous lemmas to compute the class of $K_{\overline{\cS}_2^+}$.

\begin{lem} \label{canon}
\begin{equation*} 
K_{\overline{\cS}_2^+}= \tilde{q}^*(K_ {\overline{S}_2^+}+ \frac{1}{2} \left( A_1 + B_1 \right) + \frac{1}{2} \widetilde {E}).
\end{equation*}
where $\widetilde{E} $ denotes the spin bielliptic locus  of  $\overline{S}_2^+$. 
\end{lem}
\begin{proof}
The map  $\tilde{q}$ is ramified along the locus of points in $\overline{S}_2^+ $ which admits an automorphism
group bigger than $\langle \iota \rangle$ with $\iota$ the hyperelliptic involution. It suffices to compute the codimension-one
components of this locus. 
In \cite{h} the locus of such curves have been computed for $\overline{M}_2$. It consists of the locus of bielliptic
curves plus  the boundary divisor $\Delta_1$. One immediately checks that the automorphism 
group of a general spin curve in  $\pi^{-1}(\Delta_1) = A_1 \cup B_1$
 is abelian of order 4. This together with the Lemma \ref{biellocus} proves the lemma.  
\end{proof}

From the proof we also get that the boundary class $\delta$ of the stack $\overline{\cS}_2^+$
is given by
\begin{equation} \label{eq6.2}
\delta = \tilde {q}^*(A_0 + B_0 + \frac{1}{2}A_1 + \frac{1}{2} B_1).
\end{equation}
So we get as an immediate consequence

\begin{cor} \label{ramform}
{\em (Ramification formula for $\tilde{q})$}: For every $\epsilon \in \QQ$,
\begin{equation*} 
K_{\overline{\cS}_2^+} + \epsilon \delta = \tilde{q}^* \left(K_ {\overline{S}_2^+}+
\epsilon (A_0 + B_0) +  \frac{\epsilon + 1}{2} \left( A_1 + B_1 \right) + \frac{1}{2} \widetilde {E} \right).
\end{equation*}
\end{cor}

\begin{prop} \label{ample}
 The divisor $K_{\overline{\cS}_2^+} + \epsilon \delta $ is nef (respectively ample) if
and only if  $\epsilon \geq \frac{57}{25}$ (respectively the inequality is strict).
\end{prop}

\begin{proof}
From \cite[Section 4.2 and Proposition 3.6]{h} we have 
$$
K_{\overline{M}_2} \equiv -\frac{11}{5} \Delta_0 - \frac{16}{5}\Delta_1, \quad  E \equiv 3 \Delta_0 + 12 \Delta_1.  
$$
Since by Lemma \ref{ram} we have
$\pi^*\Delta_0 = A_0 + 2B_0$ and $\pi^*\Delta_1 = A_1 + B_1$, we get
\begin{eqnarray*}
K_{\overline{\cS}_2^+} &=& \pi^*(-\frac{11}{5} \Delta_0 - \frac{16}{5}\Delta_1) + B_0\\
& = & - \frac{11}{5} A_0 - \frac{17}{5} B_0 - \frac{16}{5} A_1 - \frac{16}{5} B_1.
\end{eqnarray*}
Moreover by Lemma \ref{biellocus} and the fact that $\pi$ is \'etale of degree 10 over $M_2$ we get
\begin{eqnarray*}
\widetilde {E} & = & \frac{4}{10} \pi^*(E) \\
& = & \frac{4}{10}( \pi^*(3\Delta_0 + 12 \Delta_1)) =  \frac{6}{5}A_0 +
\frac{12}{5}B_0 + \frac{24}{5}A_1 + \frac{24}{5} B_1 
\end{eqnarray*}
So Corollary \ref{ramform} implies 
$K_{\overline{\cS_2^+}} + \epsilon \delta = \tilde{q}^* F $ with 
\begin{equation} \label{logcandiv}
F= (\epsilon - \frac{8}{5}) A_0 + (\epsilon - \frac{11}{5}) B_0 +  (\frac{\epsilon}{2} - \frac{3}{10}) A_1 + (\frac{\epsilon}{2} - \frac{3}{10}) B_1. 
\end{equation}

The inequalities of Proposition 5.3 applied to $F$ give the following conditions for the class
$F$ to be ample:
$$
\epsilon > \frac{9}{5}, \quad  \epsilon > \frac{57}{25} 
$$
(the last two inequalities do not impose conditions on $\epsilon$).
\end{proof}

Let us recall the definition of a log canonical model. Let $X$ be a normal projective variety and 
$D= \sum_{i=1}^n a_iD_i $ a $\QQ$-divisor such that $K_X + D$ is $\QQ$-Cartier and $0\leq a_i \leq 1$.

\begin{defn}
 The pair $(X,D)$ is a {\it  strict log canonical model} if $K_X + D$ is ample, $(X,D)$ has log canonical singularities,
and $X \setminus \cup_i D_i$ has canonical singularities. 
 
\end{defn}

We will also need the following proposition, whose proof we refer to 
\cite[20.2, 20.3]{ko}):
 \begin{prop}    \label{logcansing}
  Let $Y$ be a smooth variety and $f: Y \ra X$ be a finite dominant morphism to a normal variety $X$. 
Let $D = \sum_{i=1} a_iD_i $, $0\leq a_i \leq 1$ be a $\QQ$-divisor containing all the divisorial components 
of the branch locus of $f$. Let $\bar{D}$ be a $\QQ$-divisor on $Y$ such that $\supp (f^{-1}(D)) = \supp(\bar{D})$
and $f^*(K_X + D) = K_Y + \bar{D} $. Then $(X,D)$ has log canonical singularities along $D$ if and only if 
$(Y, \bar{D})$ has log canonical singularities along $\bar{D}$. 
 \end{prop}

\begin{proof}[Proof of Theorem 2(1)]
Proposition \ref{ample} shows the required ampleness of the divisor $F$. In order to verify the singularity
conditions of $\overline{S}^+_2$  we will use the results in \cite[Theorem 4.10]{h}. Consider
the finite map $\pi : \overline{S}^+_2 \ra \overline{M}_2$.  We have that $\overline{M}_2$ has canonical 
singularities away from $\Delta_0$ and $\Delta_1$ (\cite{h}). Since $\pi_{|_{S_2^+}}$ is \'etale, it follows
that   $\overline{S}^+_2$ has canonical singularities away from $\pi^{-1}(\Delta_0 \cup \Delta_1) = A_0 \cup B_0 
\cup A_1 \cup B_1$. In particular it has canonical singularities away from the boundary divisors
 $A_0 , B_0, A_1,  B_1$ and $ \widetilde{E}$. 

Since $\overline{M}_2$ has log canonical singularities along
$\Delta_0, \Delta_1$ and $E$ (\cite{h}), by Proposition \ref{logcansing}  $\overline{S}^+_2$ has 
log canonical singularities along  $\pi^{-1}(\Delta_0 \cup \Delta_1 \cup E)$ and hence also along 
$A_0 \cup B_0 \cup A_1 \cup B_1\cup \widetilde{E}$, since $\overline{M}_2$ is smooth at $E$ and $\pi$ is \'etale over $E$.
\end{proof}

\section{The invariant-theoretical compactification of $S_2^+$}

Let $C$ be a smooth projective curve of genus 2 over $k$.
Let $\FF_2^4$ denote the 4-dimensional vector space over $\FF_2$ equipped with a fixed 
symplectic form. The group of 2-division points $JC[2]$ of the Jacobian of $C$ is isomorphic to 
$\FF_2^4$ and admits a canonical symplectic form, the Weil form. 
 A {\it curve of genus 2 with a level-2 structure} 
consists of a pair $(C,\varphi)$ with a curve $C$ of genus 2 and a {\it level-2 structure} $\varphi$, that is
a symplectic isomorphism $\varphi: \FF_2^4 \ra JC[2]$.
Let $M_2(2)$ denote the coarse moduli space of such pairs. It can be constructed as follows:

The hyperelliptic covering $C \ra \PP^1$ is ramified exactly in the 6 Weierstrass points 
$p_1, \dots, p_6$ with images $x_1, \dots x_6 \in \PP^1$. It is well-known (see \cite{do}) 
that  a level-2 
structure on $C$ is equivalent to an order of the set of Weierstrass points of $C$. Denote by
$$
\Delta_6 = \{ (x_1, \dots, x_6) \in (\PP^1)^6 \;|\; x_i = x_j \; \mbox{for some} \; i \neq j \}
$$
the diagonal of the sixfold cartesian product $(\PP^1)^6$ of $\PP^1$. Then 
$$
M_2(2) \simeq \left[ (\PP^1)^6 \setminus \Delta_6 \right] //\SL_2(k)
$$
where $\SL_2(k)$ acts in the usual way on $\PP^1$ and diagonally on  
$(\PP^1)^6$. The forgetful map
$$
M_2(2) \ra M_2, \qquad (C,\varphi) \mapsto C
$$
onto the coarse moduli space $M_2$ of smooth curves of genus 2 corresponds to the quotient map
$$
\left[ (\PP^1)^6 \setminus \Delta_6 \right] //\SL_2(k) \ra \left\{ \left[ (\PP^1)^6 \setminus \Delta_6 \right] // \SL_2(k) \right\} /\mathfrak{S}_6
$$
 where the action of the symmetric group $\mathfrak{S}_6$ is induced by its permutation action on $(\PP^1)^6$.

Recall from \cite{m} that an even theta characteristic on $C$ is the line bundle given by a divisor 
$p_{i_1} + p_{i_2} - p_{i_3}$ where the $p_{i_j}$ are different Weierstrass points and 
$p_{j_1} + p_{j_2} - p_{j_3}$ defines the same theta characteristic if and only if either 
$\{j_1,j_2,j_3\} = \{i_1,i_2,i_3\}$ or $\{i_1,i_2,i_3,j_1,j_2,j_3\} = \{1, \dots ,6\}$. This implies
that the even theta characteristics of $C$ are in a natural bijective correspondence with the 3-element
subsets of the set $\{1, \dots , 6\}$ modulo the involution 
$\{i_1,i_2,i_3\} \mapsto \{1, \dots, 6\} \setminus \{i_1,i_2,i_3\}$. 
In order to construct the coarse moduli space $S_2^+$ of even spin curves of genus 2 
consider again the subgroup 
$$
G  := (\mathfrak{S}_3 \times \mathfrak{S}_3) \rtimes \langle \tau \rangle
$$
of $\mathfrak S_6$
as defined at the beginning of Section 5.

Clearly
the stabilizer of an even theta characteristic given by $\{i_1,i_2,i_3\}$ is conjugate to $G$. Hence
we obtain an isomorphism
\begin{equation} \label{eq2.1}
S_2^+ \stackrel{\simeq}{\lra}  M_2(2)/ G =  \left\{ \left[ (\PP^1)^6 \setminus \Delta_6 \right]//\SL_2(k) \right\} / G. 
\end{equation}
In \cite[p.17]{do} it is shown that a natural compactification of $M_2(2)$, the invariant-theoretical
compactification,  which we denote by 
$\overline {M_2(2)}^{inv}$, is isomorphic to the Segre cubic threefold.
In $\PP^5 = \PP^5(t_0, \dots, t_5)$ the Segre cubic is given by the equations 
\begin{equation} \label{e2.2}
s_1 := \sum_{i=0}^5 t_i = 0 \quad  \mbox{and} \quad s_3 := \sum_{i=0}^5 t_i^3 = 0
\end{equation}
where $\mathfrak{S}_6$ acts by permuting the coordinates. In other words,
\begin{equation} \label{eq2.2}
\overline {M_2(2)}^{inv} = \Proj \left( k[t_0,\dots,t_5]/(s_1,s_3) \right). 
\end{equation}
Together with \eqref{eq2.1} this implies the existence of a natural compactification $\overline{S_2^+}^{inv}$, 
the invariant-theoretical compactification of $S_2^+$, given by
\begin{equation}  \label{eq2.3}
\overline{S_2^+}^{inv}= \Proj\left( (k[t_0,\dots,t_5]/(s_1,s_3))^G \right).
\end{equation}
Here $(k[t_1,\dots,t_5]/(s_1,s_3))^G$ denotes the ring of invariants in
$\CC[t_0,\dots,t_5]/(s_1,s_3)$ where the first $\mathfrak{S}_3$ acts by permuting $t_0,t_1,t_2$, the second 
$\mathfrak{S}_3$ by permuting $t_3,t_4,t_5$ and $\tau$ by exchanging $t_i$ and $t_{i+3}$
for $i=0,1,2$.

The canonical map $M_2(2) \ra M_2$ factorizes as
\begin{equation*} 
\xymatrix{
M_2(2) \ar[rr]^f  \ar[dr]  & &S_2^+ \ar[dl]^g \\
& M_2 &
}
\end{equation*}
where $g$ is \'etale of degree 10. If we denote by $s_k = s_k(t_0,\dots,t_5) := \sum_{j=0}^5 t_j^k$ 
for  $k = 1,\dots,6$, the corresponding rings of invariants are
$$
k[t_0,\dots,t_5]/(s_1,s_3) \supset \left( k[t_0,\dots,t_5]/(s_1,s_3)\right)^G \supset 
\left( k[t_0,\dots,t_5]/(s_1,s_3) \right)^{\mathfrak{S}_6} = k[s_2,s_4,s_5,s_6]
$$
where the last equality holds since $\mathfrak{S}_6$ acts by permuting the $t_i$.
Taking the $\Proj$ of these rings gives the commutative diagram
\begin{equation} \label{diag2.5}
\xymatrix{
{\overline {M_2(2)}}^{inv} \ar[rr]^{\bar f}  \ar[dr]  & &{\overline{S_2^+}^{inv}} \ar[dl]^{\bar g} \\
& \overline{M}_2^{inv} &
}
\end{equation}
which compactifies the above diagram by \eqref{eq2.2} and \eqref{eq2.3}. 
Note that, since the ring 
extensions are finite, the maps $\bar f$ and $\bar g$ are everywhere defined and finite.
According to the following 
remark, $\overline{M}_2^{inv}$ is the classical invariant theoretical compactification of $M_2$.

\begin{rem}
With the above coordinates we have
$$
\overline{M}_2^{inv} = \Proj \left(  k[s_2,s_4,s_5,s_6] \right) = \PP(2,4,5,6).
$$
On the other hand, in terms of the invariant of binary sextics (see \cite{i} or \cite{h}),  
$$
\overline{M}_2^{inv} = \Proj( k[A,B,C,D]) = \PP(1,2,3,5)
$$
where $A,B,C,D$ are the classical invariants degree 2,4,6,10 respectively. 
By \cite[Proposition of Delorme]{d} there is a natural isomorphism 
$$
 \PP(2,4,5,6) \simeq \PP(1,2,3,5).
$$
\end{rem}

\begin{rem} \label{rem7.2}
The ring of invariants $(k[t_0,\dots,t_5]/(s_1,s_3))^G$  was explicitly computed by A. Clebsch 
in \cite[Section 61]{cl}. For a modern version see \cite[Section 4]{ksv}. We do not need the explicit form of the ring.
\end{rem}

\begin{rem}
The map $\bar g$ is of degree
$$
\deg \bar g = \frac{|S_6|}{|G|} = \frac{720}{72} = 10
$$
which coincides with the fact that every smooth curve of genus 2 admits exactly 10 even theta characteristics.
\end{rem}

In \cite{do} the variety $\overline{M_2(2)}^{inv}$ is interpreted as the moduli space of semistable 
ordered sets of 6 points in $\PP^1$. As such, the strictly semistable sets are given by the 10 singular 
points of $\overline{M_2(2)}^{inv}$. Here we consider $\overline{M_2(2)}^{inv}$ as the moduli 
space of semistable curves of genus 2 with level-2 structure. We want to determine the image of the boundary divisor
$\overline {M_2(2)}^{inv} \setminus {M_2(2)}$ in $\overline{S_2^+}^{inv}$.

Let $D_2(2), \; D_2^+$ and $D_2$ the boundary divisors of $\overline{M_2(2)}^{inv}, 
\overline{S_2^+}^{inv}$ and $\overline{M}_2^{inv}$, i.e. $D_2(2)$ is the closed subscheme 
$\overline{M_2(2)}^{inv} \setminus M_2(2)$ of $\overline {M_2(2)}^{inv}$, etc. The above two 
diagrams induce the following diagram of finite surjective maps
\begin{equation*} 
\xymatrix{
{D_2(2)} \ar[rr]^{\bar f}  \ar[dr]  & & {D_2^+} \ar[dl]^{\bar g} \\
& {D_2}&
}
\end{equation*}

We clearly have
$$
D_2(2) = \Delta_6 = \bigcup_{1 \leq i < j \leq 6} D_{ij} \quad \mbox{with} \quad D_{ij} =
\{ (x_1,\dots, x_6) \in (\PP^1)^6 \; | \; x_i = x_j \}.
$$
Denoting
$$
A_0^{inv} := \bar{f}(D_{12}) \quad \mbox{and} \quad B_0^{inv} := \bar{f}(D_{14}),
$$  
we have
\begin{prop} \label{prop7.4}
The divisor $D_2^+$ consists of the \emph{2} irreducible components
$$
D_2^+ = A_0^{inv} \cup B_0^{inv},
$$
where $A_0^{inv}$ \mbox{(}respectively $B_0^{inv}$\mbox{)} is the closure of 
points parametrizing irreducible spin curves with one node (repectively spin curves with one 
exceptional component and irreducible stable reduction).
\end{prop}

\begin{proof}
The group $G$ acts on the set of components of $D_2(2)$ with 2 orbits represented by $D_{12}$
and $D_{14}$ which gives the first assertion. The geometric interpretation of the components 
follows from \cite{k} as in \eqref{eq3.1} and \eqref{eq3.2} above. This explains also the notation.
\end{proof}

According to \cite[remark after Theorem 9.4.10, p.526]{d1} the boundary divisor $D_2(2)$ of 
$\overline {M_2(2)}^{inv} = \Proj ( k[t_0, \dots, t_5)/(s_1,s_3))$ is given by the 15 planes in the 
Segre cubic with equations
$$
t_i + t_j = t_k + t_l = t_m + t_n = 0
$$
where $\{i,j,k,l,m,n\} = \{1,\dots,6\}$.
The group $G$ acts on them with 2 orbits represented by the planes
$\Pi_1$ with equations  $ t_0 + t_3 = t_1 + t_4 = t_2 + t_5 = 0$ and $\Pi_2$ with equations
$t_0 + t_1 = t_2 + t_3 = t_4 + t_5 = 0$. The orbit of $\Pi_1$ consists of 6 planes, whereas the 
orbit of $\Pi_2$ consists of 9 planes. Comparing with \eqref{eq3.1} and \eqref{eq3.2} this implies 
$$
f^*(A_0^{inv}) = \Orb_G(\Pi_1) \quad \mbox{and} \quad  f^*(B_0^{inv}) = \Orb_G(\Pi_2)
$$
or equivalently
$$
A_0^{inv} = f(\Pi_1) \quad \mbox{and} \quad B_0^{inv} = f(\Pi_2)
$$
since the natural map $\overline {M_2(2)} \ra \overline {M_2(2)}^{inv}$  is $G$-equivariant.

The Segre cubic threefold $\overline {M_2(2)}^{inv}$ contains exactly 10 singular points i.e. nodes
(see \cite[Example 2 p.31]{do}). Their coordinates are $(\pm1, \dots, \pm1)$ where exactly half of 
them are positive. Here the equation is taken in the $\mathfrak{S}_6$-invariant form \eqref{e2.2}.
Denoting 
$$
a_1^{inv} := \overline{f} (1,1,-1,1,-1,-1) \quad \mbox{and} \quad 
b_1^{inv} := \overline{f} (1,1,1,-1,-1,-1).
$$ 
we have 

\begin{prop} \label{prop7.5}
The point $a_1^{inv}$  (respectively $b_1^{inv}$) in $\overline{S_2^+}^{inv}$
represents all stable spin curves with $2$ smooth genus-$1$ components connected by one 
exceptional
component with even (respectively odd) theta characteristics on the elliptic curves. 
\end{prop}

\begin{proof}
The group $G$ acts on 10 singular points of $\overline {M_2(2)}^{inv}$ with orbits represented by 
$(1,1,-1,1,-1,-1)$ and $(1,1,1,-1,-1,-1)$. The orbit of $(1,1,-1,1,-1,-1)$ consists of 9 singular points 
whereas $(1,1,1,-1,-1,-1)$ is a fixed point under the action. Comparing with \eqref{eq3.3} and
\eqref{eq3.4} and the definition of $A_1$ and $B_1$ gives the assertion.  
\end{proof}

\section{The map $\overline{S}_2^+ \ra \overline{S_2^+}^{inv}$}

We have constructed two compactifications of the moduli space $S_2^+$ of smooth even spin curves of genus 2,
namely the moduli space $\overline{S}_2^+$ of generalized even spin curves of 
genus 2 and the invariant theoretical compactification $\overline{S_2^+}^{inv}$. They fit into 
the following diagram
\begin{equation*} 
\xymatrix{
\overline{S}_2^+ \ar[d]_{\pi} & \overline{S_2^+}^{inv} \ar@<-1ex>[d]^<<<<{\overline g}\\
\overline{M}_2 \ar[r]^<<<<{f_2} & \overline{M}_2^{inv}
}
\end{equation*}
where $\pi$ denotes the forgetful map, $\overline g$ the map of diagram \eqref{diag2.5}, and
$f_2$ is the canonical holomorphic map constructed in \cite[Theorem 1.1]{al2}.

\begin{prop} \label{contraction}
There is a canonical holomorphic map $\overline {f_2}: \overline{S}_2^+ \ra 
\overline{S_2^+}^{inv}$ making the following diagram commutative:
\begin{equation}       \label{diag3.1}
\xymatrix{
\overline{S}_2^+ \ar[r]^>>>>>{\overline{f_2}}   \ar[d]_{\pi} & \overline{S_2^+}^{inv} \ar@<-1ex>[d]^<<<<{\overline g}\\
\overline{M}_2 \ar[r]^<<<<{f_2} & \overline{M}_2^{inv}
}
\end{equation}

Moreover, $\overline{f_2}$ contracts the 
divisor $A_1$ to the point $a_1^{inv}$, the divisor $B_1$ to the point $b_1^{inv}$, and is biholomorphic on 
${\overline S_2^+} \setminus (A_1 \cup B_1)$.
\end{prop}

\begin{proof}
Both spaces $\overline{S}_2^+$ and $\overline{S_2^+}^{inv}$ are compactifications of 
$S_2^+$, which gives a canonical birational map $\overline {f_2}$ making \eqref{diag3.1}
commutative. The boundary of $\overline{S}_2^+$ consists of the divisors $A_0,B_0,
A_1, B_1$ and the boundary of $\overline{S_2^+}^{inv}$ consists of the divisors $A_0^{inv}, 
B_0^{inv}$ and the points $a_1^{inv},b_1^{inv}$. The spaces $A_0 \setminus (A_1 \cup B_1)$
and $A_0^{inv} \setminus (a_1^{inv} \cup b_1^{inv})$  (respectively 
$B_0 \setminus (A_1 \cup B_1)$ and $B_0^{inv} \setminus (a_1^{inv} \cup b_1^{inv})$)
parametrize the same objects. Hence the map $\overline {f_2}$ extends to them. 

On the other hand, the divisors $A_1$ and $B_1$ lie over the boundary divisor $\Delta_1$ of 
$\overline {M}_2$ whereas the points $a_1^{inv}$ and $b_1^{inv}$ lie over the  boundary point 
$p_{ss}$ (see \cite[Corollary 5.3]{al1}) of $\overline {M}_2^{inv}$. Since $f_2$ contracts $\Delta_1$ 
to $p_{ss}$, the last assertion follows from Proposition \ref{prop7.5}.
\end{proof}

\section{Proof of part (2) and (3) of Theorem 2}

For the sake of abbreviation let $Y:=\overline{S_2^+}^{inv} $ and $X:= \overline{M}_2^{inv}$ 
be the coarse moduli spaces and we denote by $\tilde{q}: \cY \ra Y$ and
$q: \cX \ra X$ the corresponding moduli stacks. We have the following commutative diagram:
$$
\xymatrix{
        \cY \ar[r]^{\widetilde{q}}  \ar[d]_<<<<{\widetilde{g}}  & Y \ar[d]^<<<<{\overline {g}} \\
         \cX  \ar[r]^{q}  &  X 
    }
 $$
where $\overline{g}$ denotes the forgetful map of diagram \ref{diag3.1} and $\widetilde{g}$
the corresponding map on the stack level.
Recall from Proposition \ref{prop7.4} that the boundary of $Y$ is 
$D_2^+ =A_0^{inv} \cup B_0^{inv} $.  Similarly the boundary of $\cY$ is 
$\cD_2^+ = \alpha_0^{inv} + \beta_0^{inv}$.

Since $\Pic_{\QQ}(\overline{S}_2^+)$ is generated by $A_0, B_0, A_1, B_1$ with one relation
\eqref{eq4.5}, we deduce by means of Proposition \ref{contraction} that $\Pic(Y)$ is generated by the classes
$A_0^{inv}$ and  $B_0^{inv}$ with the relation
\begin{equation} \label{relation-inv}
3A_0^{inv} = 2 B_0^{inv}.
\end{equation}
It follows from Lemma \ref{ram} and Proposition \ref{contraction} that the map  
$\overline{g}: Y \ra X$ is simply ramified in $B_0^{inv}$ and unramified elsewhere in 
codimension 1. It is well known that $\Pic_{\QQ}(X)$ is generated by the class of the boundary 
$\Delta_0$. Hence 
$$
\overline g^* (\Delta_0) = A_0^{inv} + 2B_0^{inv}.
$$
Further we deduce from Lemma \ref{biellocus} and the fact that the bielliptic locus of $X$ is given by the class $3\Delta_0$ (see \cite{h}),  that the bielliptic locus of $Y$ is given by the class
$$
\widetilde E = \frac{2}{5} \cdot 3 \overline g^* (\Delta_0) = \frac{24}{5} A_0^{inv}.
$$
As for Corollary \ref{ramform} we obtain the following formula for $\widetilde q: \cY \ra Y$.
For every $\epsilon \in \QQ$
$$
K_{\cY} + \epsilon \cD_2^+ = \widetilde q^* \left(  \pi^*K_X + B_0 + \frac{1}{2}\widetilde{E} + \epsilon (D_2^+) \right)
$$
holds. Using the relation \eqref{relation-inv} and noting that according to \cite{h} we have
$$
K_X = - \frac{11}{5} \Delta_0,
$$
we get the following ramification formula for $\widetilde q: \cY \ra Y$.

\begin{prop} \label{positive}

For every $\epsilon \in \QQ$ we have
$$
 K_{\cY}  + \epsilon(\cD_2^+) = \tilde{q}^* \left( (- \frac{49}{10} + 
\frac{5}{2} \epsilon)A_0^{inv}  \right).
$$
\end{prop}

\begin{proof}[Proof of parts (2) and (3) of Theorem 2]

Proposition \ref{positive} implies that the class  $(- \frac{49}{10} + 
\frac{5}{2}\epsilon)A_0^{inv}$ is nef (respectively ample) if and only if $\epsilon \geq \frac{49}{25}$ (respectively $\epsilon > \frac{49}{25}$). 

It remains to check the singularity conditions of $Y$. Note first that the boundary of $Y$ is 
$A_0^{inv} \cup B_0^{inv}$. Moreover, the divisors $A_0$ and $B_0$ intersect
non-trivially in $\overline{S}_2^+$, Proposition \ref{contraction} implies that the contraction points
$a_1^{inv}$ and $b_1^{inv}$ are contained in $A_0^{inv} \cup B_0^{inv}$. Hence the proof is the same as in the proof of part (1) of Theorem 2 in Section 6.
\end{proof}

\section{Proof of part (2) of Theorem 3}

Recall that $\alpha$ is a canonical isomorphism $\alpha: {}_{\mathfrak{S}_3}{\widehat M}_2 \ra
{S_2^+}$ and ${}_{\mathfrak{S}_3}{\widetilde M}_2$ the partial 
compactification ${}_{\mathfrak{S}_3}{\widetilde M}_2 = {}_{\mathfrak{S}_3}{\widehat 
M}_2 \sqcup S$ of \cite{lo2},  with $S$ the closed subscheme of 
${}_{\mathfrak{S}_3}{\widetilde M}_2$ consisting of reducible $\mathfrak {S}_3$-coverings.
In Section 4 we proved that the construction of the map $\alpha$ does not extend to a 
holomorphic map ${}_{\mathfrak{S}_3}{\widetilde M}_2 \ra \overline{S}_2^+$. However,  we show the following

\begin{prop}\label{prop10.1}
The map $\alpha: {}_{\mathfrak{S}_3}{\widehat M}_2 \ra
{S_2^+}$ extends to a holomorphic map 
$\overline {\alpha}: {}_{\mathfrak{S}_3}{\widetilde M}_2  \ra \overline{S_2^+}^{inv}$
which contracts the divisor $S$ to the point $a_1^{inv}$ in ${S_2^+}^{inv}$.
\end{prop}

\begin{proof}
According to Lemma \ref{extension} the map $\alpha$ extends to a  birational map 
 $\tilde{\alpha}: {}_{\mathfrak{S}_3}{\widetilde M}_2 \ra S_2^+$ induced by the construction in Section  3
 defined on $ {}_{\mathfrak{S}_3}{\widehat M}_2  \cup U$ with $U \subset S$ a non-empty 
open set.  
 Since the image of the the map $\alpha$ degenerates 
 to a product of elliptic curves when a covering in ${}_{\mathfrak{S}_3}{\widehat M}_2$ degenerates
 to a spin curve in $S$, i.e. with underlying reducible curve, the image 
 $\tilde{\alpha}(U)$ is contained in  $A_1 \cup B_1$.  
 
Consider $(X_t, \kappa_t) \in S_2^+ $ a family of even spin smooth curves in the image of $\alpha$ 
degenerating to an admissible covering  $X_0 = C_1 \cup C_2$ with $C_1, C_2$ elliptic curves and 
$C_1 \cap C_2 = \{ p \}$ in the following way.
If $\kappa_t = \cO_{X_t}( \omega_{1,t}-\omega_{2,t} +\omega_{3,t})$
for some Weierstrass points $\omega_{i,t}$, $i=1,2,3$, then each element in the family comes with a 6:1 map $\psi_t: X_t \ra \PP^1$
given by the pencil $\langle 2\omega_{1,t}+2\omega_{2,t} +2\omega_{3,t}, 2\omega_{4,t}+2\omega_{5,t} +
2\omega_{6,t}  \rangle \subset |3K_{X_t}|$ (see Section 3).  The Weierstrass points $\omega_{i,t}$ specialize to points 
$\omega_{i,0} $ lying on the component $C_1$ for $i=1,2,3$ and on $C_2$ for $i=4,5,6$.  
 One checks that such admissible covering $X_0 \ra \PP^1 \cup \PP^1$ must be totally ramified at the node $p$,
 which is the neutral element for both elliptic curves 
 and the 6:1 map $C_1 \ra \PP^1$ (respectively $C_2 \ra \PP^1$ ) is determined by the pencil 
 $\langle 2\omega_{1,0}+2\omega_{2,0} +2\omega_{3,0}, 6p  \rangle$ (respectively $\langle 2\omega_{4,0}+2\omega_{5,0} +
2\omega_{6,0}, 6p \rangle$ ). Then the theta characteristic $\kappa_t$ specializes to $\kappa_0 = (\kappa_{C_1},  \kappa_{C_2})$ where $\kappa_{C_1} = \cO_{C_1}( \omega_{1,0}-\omega_{2,0} +\omega_{3,0} - p)$
and  $\kappa_{C_2} = \cO_{C_2}( \omega_{4,0}-\omega_{5,0} +\omega_{6,0} - p)$. Clearly these 2-torsion points are non-trivial, so
$(X_0, \kappa_0) \in A_1$. We conclude that $\tilde{\alpha}(U) \subset A_1$.

 Since $A_1$ is contracted to the point $a_1^{inv}$  in $\overline{S_2}^{inv}$ under the map 
$\overline f_2$, there exists a map $\bar{\alpha}$ well defined 
 on ${}_{\mathfrak{S}_3}{\widetilde M}_2$ making commutative the following diagram:
 $$
 \xymatrix@R=1.4cm@C=1cm{
{}_{\mathfrak{S}_3}{\widetilde M}_2 \ar@{-->}[r]^{\tilde{\alpha}} \ar[dr]_{\bar{\alpha}} & 
\overline{S}_2^+ \ar[d]^{\overline{f}_2} \\
 & \overline{S_2}^{inv}
}
$$
and such that $\bar{\alpha} (S)= \{ a_1^{inv}\}$.
\end{proof}

\end{document}